\documentclass[11pt]{amsart}
\usepackage{amsmath,amsfonts,amssymb,latexsym}
\usepackage{hyperref}
\hypersetup{
breaklinks=true, %permet le retour à la ligne dans les liens trop longs
urlcolor= blue, %couleur des hyperliens
}
\newtheorem{theorem}{Theorem}[section]
\newtheorem{proposition}[theorem]{Proposition}

\newcommand{\xR}{{]}{-\infty},+\infty]}
\newcommand{\Rex}{\xR}

\newcommand{\ps}{\smallbreak}

\newcommand{\lsc}{lower semicontinuous}

\newcommand{\del}{\partial}

\newcommand{\Sd}{\breve{\del}}
\newcommand{\dom} {{\rm dom} \kern.15em}
\newcommand{\tq}{:}
\newcommand{\la}{\langle}
\newcommand{\ra}{\rangle}

\newcommand{\eps}{\varepsilon}

\newcommand{\bx}{\bar{x}}
\newcommand{\xb}{\bar{x}}
\newcommand{\yb}{\bar{y}}

\begin{document}
\thispagestyle{empty}

\title
{Characterization of the monotone polar of subdifferentials}

\author{Marc Lassonde}

\address{Universit\'e des Antilles et de la Guyane,
  97159 Pointe \`a Pitre, France}

\email{marc.lassonde@univ-ag.fr}

\begin{abstract}
We show that a point is solution of the Minty variational inequality of subdifferential type
for a given function if and only if
the function is increasing along rays starting from that point.
This provides a characterization of the monotone polar of
subdifferentials of lower semicontinuous functions,
which happens  to be a common subset of their graphs depending only on the function.
\end{abstract}
%%%%%%%%%%%%%%%%%%%%%%%%%%%%%%%%%%%%%%%%%%%%%%%%%%%%%%%%%%%%%%%%%%%%%%%%

\date{July 6, 2013}



\maketitle

%%%%%%%%%%%%%%%%%%%%%%%%%%%%%%%%%%%%%%%%%%%%%%%%%%%%%%%%%%%%%%%%%%%%%%%
Given a graph $T\subset X\times X^*$, we call \textit{monotone polar of $T$} the set
$$
T^0:=\{(x,x^*)\in X\times X^*\tq\langle y^*-x^*,y-x\rangle \geq 0,\quad\forall (y,y^*)\in T\}.
$$
The object of this note is to characterize the monotone polar of
subdifferentials $\del f$ of extended real-valued \lsc\ functions $f$ on a Banach space $X$.

We see from the above definition that
a pair $(\xb,0)$ in $X\times X^*$ belongs to $(\del f)^0$ if and only if
$\xb$ is solution of the \textit{Minty variational inequality of subdifferential type}
defined by:
$$\sup \la \del f(y),\xb-y\ra \le 0,\quad\forall y\in X.$$
We are thus led to consider such systems of inequalities.
\ps
In Section \ref{MVIderiv}, we study Minty variational inequalities of
(lower Dini) subderivative type for functions $f$ on a convex subset $C\subset X$.
We show that a point $\xb\in C$ is solution of such a system
if and only if $f$ is increasing on $C$ along rays starting from $\bx\in C$.
(This is a variant of \cite[Theorem 2.1]{CGR04}.)
In Section \ref{MVIdiff}, we consider the case of subdifferentials of functions $f$.
These objects are simply set-valued mappings $\partial f: X \rightrightarrows X^\ast$
that lie between the subdifferential of convex analysis for $f$ and its Clarke subdifferential,
and satisfy the so-called Separation Principle.
Using the link established in \cite{JL13} between subderivative and subdifferential,
we show that Minty variational inequalities of subderivative type and of
subdifferential type have the same set of solutions.
Finally, in Section \ref{polar}, we provide the desired characterization of
the monotone polar of subdifferentials  of functions $f$: it consists in
the set of all pairs $(x,x^*)$ in $X\times X^*$ such that the perturbed function
$f-x^*$ is increasing on $X$ along rays starting from $x$.
Therefore, all subdifferentials have the same polar, which is a common
subset of their graphs.

%%%%%%%%%%%%%%%%%%%%%%%%%%%%%%%%%%%%%%%%%%%%%%%%%%%%%%%%%%
\section{Minty variational inequality of subderivative type}\label{MVIderiv}

In the following, $X$ is a real Banach space with unit ball $B_X$,
$X^*$ is its topological dual,
and $\la .,. \ra$ is the duality pairing.
For $x, y \in X$, we let $[x,y]:=\{ x+t(y-x) \tq t\in[0,1]\}$, and $[x,y[$ is defined accordingly.

All the functions $f : X\to\Rex$ are assumed to be lower semicontinuous
and \textit{proper}, which means that
the set  $\dom f:=\{x\in X\tq f(x)<\infty\}$ is nonempty.
The (radial or lower Dini) \textit{subderivative} of a function $f$ at a point
$\bx\in \dom f$ is given by:
\begin{equation*}\label{general-dd}
\forall d\in X,\quad f'(\xb;d):=\liminf_{t\searrow 0}\,\frac{f(\xb+td)-f(\xb)}{t}.
\end{equation*}

Here is a simple version of the mean value inequality for extended real-valued \lsc\ functions
in terms of the subderivative
(for a proof, see, e.g., \cite{JL12,JL13}):

\medbreak\noindent
{\bf Mean Value Inequality}.
{\it Let $X$ be a Banach space, $f:X\to\xR$ be \lsc, $\xb\in X$ and $x\in\dom f$. Then, for every real number $\lambda\le f(\xb)-f(x)$, there exists $x_0\in [x,\xb[$ such that
$\lambda\le f'(x_0;\xb-x)$. 
}
\medbreak
The next proposition asserts that, given a function $f$ and a convex set $C\subset X$,
a point $\xb\in C$
is solution of the Minty variational inequality of subderivative type
associated to $f$ on $C$
if and only if
$f$ is increasing on $C$ along rays starting from $\xb$.
It is a variant of \cite[Theorem~2.1]{CGR04}.

\begin{proposition}\label{mainprop}
Let $X$ be a Banach space, $f:X\to\xR$ be \lsc, $C\subset X$ be convex and $\xb\in C$.
Then, the following are equivalent:\ps
{\rm(a)} $f'(y;\xb-y) \le 0$ for every $y\in C$,\ps
{\rm(b)} $f(y+t(\xb-y))\le f(y)$ for every $y\in C$ and $t\in[0,1]$.
\end{proposition}
\begin{proof}
We show that $\neg$(b) implies $\neg$(a). 
Assume there exist $y\in C$ and $\yb\in [y,\xb]$ % =y+t(\xb-y)$ with $t\in[0,1]
such that $f(\yb)>f(y)$.
Applying the above Mean Value Inequality with $0<\lambda<f(\yb)-f(y)$, we find
$y_0\in [y,\yb[\subset C$ such that
\[f'(y_0;\yb-y)\ge \lambda> 0.\]
Since $\xb-y_0=\tau(\yb-y)$ for some $\tau>0$, we have $f'(y_0;\xb-y_0)=\tau f'(y_0;\yb-y)> 0$,
proving that $\neg$(a) holds.
\ps
Conversely, we show that (b) implies (a).
Let $y\in C$ and assume $f(y+t(\xb-y))\le f(y)$ for every
$t\in [0,1]$. Then,
$$
\limsup_{t\searrow 0}\frac{f(y+t(\xb-y))- f(y)}{t}\le 0;
$$
a fortiori, $f'(y;\xb-y)\le 0$.
\end{proof}

%%%%%%%%%%%%%%%%%%%%%%%%%%%%%%%%%%%%%%%%%%%%%%%%%%%%%%%%%%
\section{Minty variational inequality of subdifferential type}\label{MVIdiff}

The subdifferential of a \lsc\ function $f$ in the sense of convex analysis
%is the set-valued operator $\partial_{cx} f: X \rightrightarrows X^\ast$ given by
at a point $\bx\in \dom f$ is the subset of $X^*$ defined by
$$
\del_{cx} f(\xb):= \{ x^* \in X^* \tq \la x^*,y-\xb\ra + f(\xb) \leq f(y),\, \forall y \in X \}.
$$
The Clarke subdifferential \cite{Cla90} of $f$ at $\bx\in \dom f$ is given by
\begin{eqnarray*}\label{Csub}
\partial_{C} f(\bx) := \{x^* \in X^* \tq \langle x^*,d\rangle \leq
f^{\uparrow}(\bx;d), \, \forall d \in X\},
\end{eqnarray*}
where
\begin{eqnarray*}\label{Csubderiv}
f^{\uparrow}(\bx;d):= \sup_{\delta>0}
\limsup_{\substack{t \searrow 0 \\{x \to_f \bx}}} \inf_{d' \in d+ \delta B_X}
\frac{f(x+td') -f(x)}{t}.
\end{eqnarray*}
It is easily seen that $\del_{cx} f(\xb)\subset \partial_{C} f(\bx)$
with equality whenever $f$ is convex \lsc.
\ps
In what follows, we call \textit{subdifferential} any operator $\del$ that associates 
a set-valued mapping $\partial f: X \rightrightarrows X^\ast$
to each function $f$ on $X$ so that $\partial f$ lies between $\del_{cx} f$
and $\partial_{C} f$:
\begin{equation}\label{inclusdansClarke}
\del_{cx} f\subset \partial f\subset \partial_{C} f,
\end{equation}
and satisfies the elementary property
$\del(f-x^*)(x)=\del f(x) -x^*$ for all $x\in X$ and $x^*\in X^*$.
We also require that subdifferentials satisfy
the following basic calculus rule on the Banach space $X$: 
\medbreak
\textit{Separation Principle}.
For any \lsc\ functions $f,\varphi$ on $X$ with $\varphi$ convex Lipschitz
near $\xb\in\dom f \cap\dom \varphi$,
if $f+\varphi$ admits a local minimum at $\xb$, then
$0\in \del f(\xb)+ \del \varphi(\xb).$
\ps
\textit{Examples}.
The Clarke subdifferential and %the Michel-Penot subdifferential,
the Ioffe subdifferential satisfy the Separation Principle in any Banach space.
The limiting versions of the basic elementary subdifferentials
(proximal, Fr\'echet, Hadamard) satisfy the Separation Principle
in appropriate Banach spaces.
All these subdifferentials also satisfy the inclusions (\ref{inclusdansClarke})
For more details, see, e.g., \cite{Iof12,JL12} and the references therein.
\ps
Here is the link between subderivative and subdifferential.
It involves the following $\eps$-enlargement of the subdifferential:
\begin{equation*}\label{FJ-sdiff}
\Sd_\eps f(\xb):=\{x^*_\eps\in X^*\tq x^*_\eps\in\del f(x_\eps) \mbox{ with }
\|x_\eps-\xb\|\le\eps,\,|f(x_\eps)-f(\xb)|\le\eps,\, \la x^*_\eps,x_\eps-\xb\ra\le \eps\}.
\end{equation*}

\medbreak\noindent
{\bf Subderivative/Subdifferential Inequality \cite{JL13}}.
{\it Let $X$ be a Banach space, $f:X\to\xR$ be lower semicontinuous
and $\xb\in\dom f$. Then, for every $\eps>0$,
the sets $\Sd_\eps f(\xb)$ are nonempty and
\begin{equation}\label{CDD-formula}
\forall d\in X,\quad f'(\xb;d)\le\inf_{\eps>0}\sup\la \Sd_\eps f(\xb),d\ra.
\end{equation}
}
\smallbreak
Injecting Formula (\ref{CDD-formula}) into  Proposition \ref{mainprop}
enables us to show that
Minty variational inequalities of subderivative type and of subdifferential type
on an open convex set have actually the same set of solutions.
This complements our previous result \cite[Theorem 3.3]{JL13}
on optimality conditions in terms of subdifferentials.

\begin{theorem}\label{main}
Let $X$ be a Banach space, $f:X\to\xR$ be \lsc, $U\subset X$ be
open convex and $\xb\in U$.
Then, the following are equivalent:\ps
{\rm(a)} $\,\sup \la \del f(y),\xb-y\ra \le 0$ for every $y\in U$,\ps
{\rm(b)} $f(y+t(\xb-y))\le f(y)$ for every $y\in U$ and $t\in[0,1]$.
\end{theorem}
\begin{proof}
We show that $\neg$(b) implies $\neg$(a).
Assume there exist $y \in U$ and $\yb\in[y,\xb]$
such that $f(\yb)> f(y)$. Then, by Proposition \ref{mainprop}
there exists $y_0\in U$ such that $f'(y_0;\xb-y_0) >0$.
Let $\eps>0$ such that $f'(y_0;\xb-y_0) > \eps$ and $y_0+\eps B_X\subset U$,
and apply Formula (\ref{CDD-formula}) at point $y_0$
and direction $d=\xb-y_0$ to
obtain a pair $(y_\eps,y^*_\eps)\in \del f$ such that
\[
\|y_\eps-y_0\|< \eps, %\quad f(y_\eps)<f(x_0)+\eps,
\quad\la y^*_\eps, y_\eps-y_0\ra <\eps,
\quad\la y^*_\eps, \xb -y_0\ra >\eps.
\]
Hence, $y_\eps\in U$ and $y^*_\eps\in \del f(y_\eps)$
satisfy $\la y^*_\eps,\xb -y_\eps\ra>0$,
which proves $\neg$(a).

Conversely, let us assume that $f(y'+t(\xb-y'))\le f(y')$ for every $y'\in U$ and
$t\in [0,1]$. We have to show that
$\sup \la \del f(y),\xb-y\ra \le 0$ for every $y\in U$.
% under the additional assumption that $\del f\subset \del_Cf$.
Let $y\in U$. We claim that $f^{\uparrow}(y;\xb-y)\le 0$.
Indeed, let $\delta>0$ such that $y+\delta B_X\subset U$, $t\in ]0,1]$
and $y'\in y+\delta B_X$. Then
$$
\inf_{d'\in \xb-y+\delta B_X}\frac{f(y'+td') -f(y')}{t}
  \le \frac{f(y'+t(\xb-y')) -f(y')}{t}\le 0.
$$
It follows that, for any $\delta>0$,
$$
\limsup_{\substack{t \searrow 0 \\{y' \to_f y}}}
\inf_{d'\in \xb-y+\delta B_X}\frac{f(y'+td') -f(y')}{t}\le 0,
$$
hence, $f^{\uparrow}(y;\xb-y)\le 0$. This proves the claim.
We derive that
$$\sup \la \del_C f(y),\xb-y\ra \le f^{\uparrow}(y;\xb-y)\le 0,$$
so also $\sup \la \del f(y),\xb-y\ra \le 0$ since $\del f\subset \del_Cf$.
This completes the proof.
\end{proof}

%%%%%%%%%%%%%%%%%%%%%%%%%%%%%%%%%%%%%%%%%%%%%%%%%%%%%%%%%%
\section{Monotone polar of subdifferentials}\label{polar}

Given a set-valued operator $T:X\rightrightarrows X^*$,
or graph $T\subset X\times X^*$, we let
$$
T^0:=\{(x,x^*)\in X\times X^*\tq\langle y^*-x^*,y-x\rangle \geq 0,\ 
\forall (y,y^*)\in T\}
$$
be the set of all pairs $(x,x^*)\in X\times X^*$ {\it monotonically related} to $T$.
As in \cite{M-LS05}, we call \textit{monotone polar of $T$}
the operator $T^0:X\rightrightarrows X^*$, or graph $T^0\subset X\times X^*$, thus defined.
An operator $T:X\rightrightarrows X^*$ is said to be
{\it monotone} provided $T\subset T^0$,
{\it maximal monotone} provided $T= T^0$ and
{\it monotone absorbing} provided $T^0\subset T$.

In \cite{JL13} we proved that all subdifferentials of \lsc\ functions,
not only the convex ones, are monotone absorbing.
Below we show a more precise statement:
the monotone polar of any subdifferential $\del f$
%contained in the Clarke subdifferential have the same monotone polar, namely
is precisely the set of all pairs $(x,x^*)\in X\times X^*$ such that the perturbed function
$f-x^*$ is increasing along all rays starting from $x$.
%This is the object of the next theorem:
Thus, all subdifferentials of \lsc\ functions have the same monotone polar
and are monotone absorbing. 

\begin{theorem}\label{main2}
Let $X$ be a Banach space and let $f:X\to\xR$ be proper \lsc. Then, for every $x\in X$,
$$
(\del f)^0(x)=\{\,x^*\in X^*\tq (f-x^*)(y+t(x-y))\le (f-x^*)(y), \ 
\forall y \in X,\ t\in[0,1]\,\}.
$$
\end{theorem}
\begin{proof}
Let $x^*\in (\del f)^0(x)$. This means that
$\la y^*-x^*,y-x\ra\ge 0$ for every $(y,y^*)\in\del f$.
Since $\del(f-x^*)(y)=\del f(y) -x^*$,
this can be rewritten as
$$
\sup\, \la \del (f-x^*)(y),x-y\ra\le 0 \quad\mbox{for every } y\in X,
$$
which is equivalent by Theorem \ref{main} to
$$(f-x^*)(y+t(x-y))\le (f-x^*)(y)\quad\mbox{for every } y\in X
\mbox{ and } t\in[0,1].
$$
The proof is complete.
\end{proof}

\end{document}